\documentclass[11pt,draftcls,onecolumn]{IEEEtran}
\usepackage{graphicx}
\usepackage{amssymb}
\usepackage{amsmath}
\usepackage{epstopdf}
\usepackage{dsfont}
\usepackage{amsthm}
\usepackage{xcolor}
\usepackage{colortbl}
\usepackage{algorithm}
\usepackage{algorithmic}
\usepackage{todonotes}
\usepackage{booktabs}
\usepackage{afterpage}
\usepackage{mathtools}
\usepackage{multirow}
\usepackage{stfloats}

\newcommand {\be}{\begin{equation}}
\newcommand {\ee}{\end{equation}}

\newcommand {\beW}{\begin{equation*}}
\newcommand {\eeW}{\end{equation*}}

\newcommand {\bsp}{\begin{split}}
\newcommand {\esp}{\end{split}}

\newcommand {\bea}{\begin{eqnarray}}
\newcommand {\eea}{\end{eqnarray}}

\newcommand {\beaW}{\begin{eqnarray*}}
\newcommand {\eeaW}{\end{eqnarray*}}

\newcommand {\bM}{\begin{bmatrix}}
\newcommand {\eM}{\end{bmatrix}}

\newcommand {\bi}{\begin{itemize}}
\newcommand {\ei}{\end{itemize}}

\newcommand {\ben}{\begin{enumerate}}
\newcommand {\een}{\end{enumerate}}

\newcommand {\bal}{\begin{align}}
\newcommand {\eal}{\end{align}}

\newcommand {\balW}{\begin{align*}}
\newcommand {\ealW}{\end{align*}}

\newcommand {\bcm}{\begin{columns}}
\newcommand {\ecm}{\end{columns}}
\newcommand {\bc}{\begin{column}}
\newcommand {\ec}{\end{column}}

\newtheorem{thm}{Theorem}
\newtheorem{prop}{Proposition}
\newtheorem{corol}{Corollary}

\newtheorem{assumption}{Assumption}

\DeclareMathOperator{\Xs}{\mathcal{X}}

\DeclareMathOperator{\Prb}{\mathbb{P}}

\DeclareMathOperator{\eps}{\varepsilon}
\DeclareMathOperator{\rcp}{\text{RCP}}

\DeclareMathOperator{\scp}{\text{SCP}}
\DeclareMathOperator{\ccp}{\text{CCP}}

\let\Re\relax
\DeclareMathOperator{\Re}{\mathbb{R}}
\DeclareMathOperator{\Ne}{\mathbb{N}}

\DeclarePairedDelimiter\parens{\lparen}{\rparen}
\DeclarePairedDelimiter\bracks{\{}{\}}
\DeclarePairedDelimiter\sqbracks{[}{]}

\renewcommand{\baselinestretch}{1.111}

\date{\today}
\title{On the computational complexity and generalization properties of multi-stage and recursive scenario programs}

\author{Nikolaos~Kariotoglou, Kostas~Margellos and John~Lygeros
\thanks{N. Kariotoglou and J. Lygeros are with the Automatic Control Laboratory, Department of Information Technology and Electrical Engineering, ETH Z\"urich, Z\"urich 8092, Switzerland (e-mail: karioto@control.ee.ethz.ch; lygeros@control.ee.ethz.ch)}\thanks{K. Margellos is with the Department of Industrial Engineering and Operations Research, UC Berkeley, Sutardja Dai Hall 330, Berkeley CA 94720, United States (e-mail: kostas.margellos@berkeley.edu)}\thanks{The work of N. Kariotoglou was supported by the Swiss National Science Foundation under grant number $200021\_137876$.}\vspace{-1\baselineskip}}

\begin{document}
\maketitle
\begin{abstract}
We discuss the computational complexity and feasibility properties of scenario based techniques for uncertain optimization programs. We consider different solution alternatives ranging from the standard scenario approach to recursive variants, and compare feasibility as a function of the total computation burden. We identify trade-offs between the different methods depending on the problem structure and the desired probability of constraint satisfaction. Our motivation for this work stems from the applicability and complexity reduction when making decisions by means of recursive algorithms. We illustrate our results on an example from the area of approximate dynamic programming.
\end{abstract}

\section{Introduction}

Robust optimization comes up naturally in a range of problems from finance to robotics (\cite{kouvelis1997robust,bertsimas2011theory,ben2002robust}). Uncertain data is often present in the formulation of a decision making problem and the optimal solution is required to be robust against any possible uncertainty realization \cite{ben2009robust}. However, uncertainty may take values from an infinite and possibly unbounded set, which we might not know analytically, giving rise to a robust optimization problem that is in general not tractable \cite{calafiore2000randomized,calafiore2007survey,petersen2014robust}. A significant amount of research has concentrated on robust problems with structural characteristics and uncertainty sets of specific geometry for which robust decisions can be made by means of a tractable optimization program \cite{bertsimas2004price}.

An alternative way to deal with data uncertainty is to
formulate a chance constrained variant of the initial problem where the optimal decision is allowed to violate the robust constraint on a set of pre-specified measure. The authors in \cite{ben2009robust,bertsimas2006tractable} provide explicit solutions to such problems under assumptions on the probability distribution of the uncertainty. To avoid such assumptions one can make use of uncertainty samples (either based on historical data or via a scenario generation model) and construct decisions that satisfy the system constraints only for the sampled uncertainty scenarios. The feasibility and performance properties of the solution can be generalized to quantify the confidence with which the optimizer of the scenario program satisfies the constraints for uncertainty realizations different than those used in the optimization process, providing a probabilistic link between scenario based and chance constrained optimization. The \textit{scenario approach} introduced in \cite{calafiore2006,campi2008exact} can be used to provide such feasibility generalization statements for convex optimization problems. Beyond feasibility guarantees, \cite{LecLygMac_2010,kanamori2012worst,mohajerin2014performance} provide bounds on the amount of constraint violation and probabilistic performance. Generalization properties of similar nature can be obtained for non-convex optimization programs as well, using VC theoretic results \cite{vidyasagar2002theory,tempo2004algorithms,alamo2010sample}; the complexity, however, of the resulting solution depends on the so-called VC dimension (see \cite{vidyasagar2002theory} for a precise definition), which is in general difficult to compute.

Here we focus on the generalization properties of scenario based convex optimization problems using the scenario approach \cite{calafiore2006,campi2008exact,georg2013,calafiore2013random}. The scenario approach deals with robust and chance constrained convex optimization problems by solving sampled programs constructed using a finite number of samples. The method provides bounds on the number of samples needed to provide guarantees about the feasibility of the optimal solution of the sampled program with respect to the original one. The number of required samples determines the total number of constraints that, together with the number of decision variables and the type of problem (linear program, quadratic program, second-order cone program, semi-definite program, etc.), determine the overall computation effort. The given bounds scale well with certain structural quantities of the underlying problem and with the design parameters, apply to any problem under relatively mild assumptions and can be shown to be tight in a specific class of problems \cite{campi2008exact, calafiore2010random}. However, when one considers optimization problems with additional structure on the constraints, the generic bounds on the number of samples are not a sufficient performance measure. The same guarantees on the feasibility of a scenario based solution may be obtained by formulating several alternative scenario programs, each with a potentially different number of decision variables and constraints and hence different computational complexity. Here we investigate these trade-offs for a class of recursive optimization problems that naturally arises in applications such as stochastic model predictive control (SMPC) \cite{calafiore2013robust} and approximate dynamic programming (ADP) \cite{kariotoglou2014adp}. We consider two alternative structures, one with a single convex optimization problem with multiple constraint functions and one where the constraint functions are coupled. We show how, besides the standard scenario program, other types of scenario programs can be formulated for generic problems that respect these structures. These alternatives provide the same feasibility guarantees at potentially lower computation cost. We demonstrate this trade-off by benchmarking on a particular class of algorithms (primal-dual) and a particular class of problems (robust second-order cone problems). We also show how the stage-wise confidence and the violation level, typically treated as parameters in scenario programs, can be chosen by means of a convex optimization program to reduce the overall computation time. We demonstrate our results by applying them to a particular ADP algorithm developed for reachability problems.

Section \ref{sec:set_up} provides a statement of the problem under consideration. In Section \ref{sec:gen} we present the different scenario based alternatives along with a pair of convex optimization problems that choose the stage-wise confidence and violation probability levels that result in the most favorable computational complexity. In Section \ref{sec:trade} we discuss the trade-off between the feasibility properties and the computational complexity of each alternative. Section \ref{sec:applications} illustrates some features of the different algorithmic alternatives by means of a numerical example arising in ADP, while Section \ref{sec:conclusion} concludes the paper with some ideas on future research directions.

\emph{Notation:} Let $\mathbb{R}$ denote the real numbers, $\mathbb{N}$ the natural numbers and $\mathbb{N}_+$ the positive natural numbers. In the derivations below all uncertainty samples are extracted from a (possibly unknown) uncertainty set $\Delta$ according to a fixed, possibly unknown probability measure $\mathbb{P}$. $\mathbb{P}^S$ denotes the corresponding product measure for some $S \in \mathbb{N}_+$. We use i.i.d for identically and independently distributed uncertainty samples.
Operator $|\cdot|$ denotes the cardinality of its argument, $\dim(A)$ denotes the dimension of a linear space $A$ and $x \models y$ implies that $x$ satisfies the statement in $y$.

\section{Convex optimization programs with multiple robust constraints} \label{sec:set_up}

Consider a compact convex set $\mathcal{X}\subseteq \Re^d$, a possibly unbounded uncertainty set $\Delta\subseteq\Re^w$, a convex cost function $f:\Xs\rightarrow \Re$ and a set of $M\in\Ne_+$ convex constraint functions $g_i:\Xs \times \Delta \rightarrow \Re$, $i=1,\dots,M$; our results also extend to non-real valued (e.g. binary valued) uncertainties as long as they take values in a probability space. We are concerned with robust convex optimization problems (RCP) of the form:
\begin{align}
\rcp:
\begin{cases}
	\begin{split}
	\min_{x\in\mathcal{X}}\quad &f(x)\\
	 \text{s.t} \quad &g_i\left(x,\delta\right) \leq 0,\ \forall \delta\in\Delta, \ \forall i\in\{1,\dots,M\}.
	\end{split}
	\end{cases}
	\label{eq:rob_opt}
\end{align}

The set $\Delta$ may be infinite and possibly unbounded, rendering (\ref{eq:rob_opt}) a convex, semi-infinite optimization program. For such problems there is no general algorithm to obtain a solution, unless particular assumptions on the structure of $\Delta$ and the functions $g_i$ are made (see for example \cite{bertsimas2004price,ben1999robust}).

A common approach to approximate the solution is to impose the constraints on a finite number of uncertainty instances. To this end, consider $S \in\mathbb{N}_+$ i.i.d samples $\{\delta^j\}_{j=1}^S$ extracted from $\Delta$ according to some, possibly unknown, underlying probability distribution, and a collection $\bracks*{\Delta_i}_{i=1}^M$ of $M$ subsets of $\{\delta^j\}_{j=1}^S$ such that for each $\delta \in \{\delta^j\}_{j=1}^S$ there exists $i$ so that $\delta \in \Delta_i$, i.e. the sets may be overlapping but each $\delta$ belongs in at least one of them. The interpretation is that for each $i=1,\ldots,M$, the corresponding constraint $g_i(x,\delta)$ should be satisfied for all $\delta \in \Delta_i$, but not necessarily for all $\delta \in \Delta$. Problem (\ref{eq:rob_opt}) is then approximated by a scenario convex optimization program ($\text{SCP}$) of the form:
\begin{align}
\text{SCP}\sqbracks*{\Delta_1,\dots,\Delta_M}:
\begin{cases}
	\begin{split}
	\min_{x\in\mathcal{X}}\quad &f(x)\\
	 \text{s.t} \quad &g_i\left(x,\delta\right) \leq 0,\forall \delta\in\Delta_i, \ \forall i\in\{1,\dots,M\}.
	\end{split}
	\end{cases}
	\label{eq:rob_opt_samp}
\end{align}
This is a convex optimization program with a finite number of decision variables and constraints, that can be solved to optimality by various numerical solvers (e.g. CPLEX, Gurobi, MOSEK).
We impose the following assumption on $\text{SCP}[\Delta_1,\dots,\Delta_M]$:
\begin{assumption} \label{ass:feas_uniq}
For any set $\{\delta^j\}_{j=1}^S$ and collection of subsets $\bracks*{\Delta_i}_{i=1}^M$ with $S,M\in \mathbb{N}_+$, $\text{SCP}[\Delta_1,\dots,\Delta_M]$ is feasible, its feasibility region has a non-empty interior and its minimizer $x^*[\Delta_1,\ldots,\Delta_M] :~ \Delta^{S} \rightarrow \mathcal{X}$ is unique.
\end{assumption}

We refer to \cite{calafiore2006}, \cite{calafiore2010random} for details on how the feasibility and uniqueness assumption can be relaxed; however, we keep these assumptions here to streamline the presentation of our results. Measurability of the minimizer $x^*\sqbracks*{\Delta_1,\dots,\Delta_M}$ is assumed as needed, see \cite{mohajerin2014performance,grammatico2014} for details.

Note that problem (\ref{eq:rob_opt_samp}) and its minimizer are parametrized by the sets $\Delta_1,\dots,\Delta_M$. However, once the sets $\Delta_i$ are fixed, the unique (under Assumption \ref{ass:feas_uniq}) minimizer $x^* [\Delta_1,\dots,\Delta_M]$ of $\text{SCP}[\Delta_1,\dots,\Delta_M]$ is a mapping from $\Delta^{S}$ to $\mathcal{X}$ and satisfies $g_i(x^*,\delta)\leq 0$, for all $\delta\in\Delta_i$ and  $i\in\left\{1,\dots,M\right\}$. One of the challenges concerning this type of problem is to analyze the properties of $x^*$ in terms of feasibility (satisfiability of $g_i(x^*,\delta)\leq 0$ for all $\delta \in \Delta$) and performance (optimality of $f(x^*)$). Following the standard literature on the scenario approach \cite{calafiore2006,campi2008exact} we concentrate here on the feasibility properties of $x^*$ as a function of the algorithm used to construct the solution; for a discussion on performance issues see \cite{kanamori2012worst, mohajerin2014performance}. 
We establish that, depending on problem structure, there may be different ways of formulating the scenario program as a function of the choice of the number of samples $S$ and the partition sets $\Delta_i$. The computation effort necessary to solve the corresponding scenario programs differs, despite the fact that solutions have comparable feasibility properties. We will investigate the trade-offs between different design choices in the context of second order cone problems (SOCP) solved via primal-dual algorithms which are in general known to be of $\mathcal{O}\parens*{\parens*{n+m}^3}$ complexity, where $n$ denotes the dimension of the decision space and $m$ the total number of constraints. Our motivation stems from the fact that a wide range of control-inspired optimization programs are SOCP, while primal-dual algorithms provide reliable termination and optimality conditions by iteratively reducing the duality gap. With straightforward modifications, related statements can be made for other classes of algorithms (e.g. gradient methods) and other classes of problems (linear programs, quadratic programs, semi-definite programs etc.).

\section{Feasibility properties of scenario convex programs} \label{sec:gen}

We introduce four different approaches to formulate the scenario program: the standard scenario approach, the multi-stage scenario approach, the recursive scenario approach using the same samples at every recursive step and the recursive scenario approach using different samples at every recursive step. The standard scenario approach is the most general and applies to all problems in the form of $\rcp$. The multi-stage and recursive counterparts assume particular structure on the constraint functions and exploit it to reduce computational complexity while maintaining similar feasibility properties. To streamline the comparison between different methods, we present here their main characteristics and devote the next section to discussing relative advantages. 
\subsection{The standard scenario approach}
\label{sec:scp}
Let $\bar{\Delta} = \{\delta^j\}_{j=1}^S$ and assume that $\Delta_1=\ldots=\Delta_M=\bar{\Delta}$, in other words enforce each constraint for all elements in $\bar{\Delta}$. Denote by $\text{SCP} [\bar{\Delta}]$, $x^* [\bar{\Delta}]$ the resulting instance of
$\text{SCP}[\Delta_1,\dots,\Delta_M]$ and its minimizer, respectively. For each $(x, \delta)$, let $g(x,\delta):=\max_{i=1,\ldots,M}g_i(x,\delta)$. The constraints $g_i\left(x,\delta\right) \leq 0, \ \forall \delta\in \bar{\Delta},\ \forall i\in\{1,\dots,M\}$ are equivalent to $g(x,\delta) \leq 0, \ \forall \delta\in \bar{\Delta}$.
Problem $\text{SCP} [\bar{\Delta}]$ was first studied in terms of feasibility in \cite{calafiore2006}. The following Theorem was then shown in \cite{campi2008exact}.

\begin{thm}[\cite{campi2008exact}, Theorem 2.4]
\label{thm:scp}
Choose $\varepsilon, \beta \in (0,1)$ and fix $S\geq S(\eps,\beta,d)$ where
\begin{align}
	\begin{split}
	S(\varepsilon, \beta, d):=\min \left\{ N\in\Ne \ \bigg | \ \sum_{i=0}^{d-1} \binom{N}{i}\varepsilon^i(1-\varepsilon)^{N-i} \leq \beta\right\}.
	\end{split}
	\label{eq:sample_bound}
\end{align}
Extract $S$ samples i.i.d from $\Delta$ according to a probability measure $\Prb$, construct $\Delta_1 = \ldots = \Delta_M = \bar{\Delta}$ and formulate $\scp[\bar{\Delta}]$. Under Assumption \ref{ass:feas_uniq}, the minimizer $x^*[\bar{\Delta}]$ of $\scp[\bar{\Delta}]$ satisfies the chance constraint,
\begin{align}
\ccp_{\eps}:
\Prb\sqbracks*{g\parens*{x^*[\bar{\Delta}],\delta} > 0} \leq \eps
\label{eq:echance}
\end{align}
with confidence (measured with respect to $\mathbb{P}^S$) at least $1-\beta$.
\end{thm}
Using the satisfiability  notation ``$\models$'' (along the lines of \cite{mohajerin2014performance}), the statement in Theorem \ref{thm:scp} can be compactly written as $\Prb^{S}\sqbracks*{x^*[\bar{\Delta}]\models \ccp_{\eps} }\geq 1-\beta$.
The interpretation of Theorem \ref{thm:scp} is that, with certain confidence, the solution of the scenario convex program satisfies the robust constraint apart from a subset of the uncertainty space with measure at most $\eps$. The computational complexity associated with constructing $x^*[\bar{\Delta}]$, along with the feasibility properties of Theorem \ref{thm:scp}, depend on the choice of $\eps,\beta$ and the number of decision variables $d$ that implicitly affect the number of constraints (inspect \eqref{eq:sample_bound}). The exact effect of each parameter is discussed in Section \ref{sec:trade}. Note that Theorem \ref{thm:scp} remains unaffected if $d$ is replaced by any upper bound on the number of the so-called support constraints (see \cite{calafiore2006} for a precise definition) other than the dimension of the decision space. Refinements along this direction are discussed in \cite{georg2013,georg2014smpc,zhang2014structured}.

\subsection{The multi-stage scenario approach}
\label{sec:ext_scp}
We impose here additional structure on the $\rcp$ by assuming that for any $\delta \in \Delta$ and for each $i=1,\ldots,M$, the constraint function $g_i(\cdot,\delta)$ depends on some (i.e. not necessarily all) of the decision variables. The set-up is then similar to the structure considered in \cite{georg2013}, where the authors studied optimization programs with multiple chance constraints. For each $i=1,\ldots,M$, let $\mathcal{X}_i\subseteq \Xs$ denote the domain of each $g_i(\cdot,\delta)$ and $d_i  = \dim (\mathcal{X}_i)$, where $\dim (\mathcal{X}_i)$ denotes the dimension of the smallest subspace of $\mathbb{R}^d$ containing $\Xs_i$. We further assume that 
$d_i<d$ for at least one $i=1,\ldots,M$ to exclude the case where all constraint functions depend on all the decision variables; if this is not the case the subsequent analysis reduces to the standard scenario approach of Section \ref{sec:scp}. We then have the following theorem due to \cite{georg2013}, which serves as the multi-stage counterpart of Theorem \ref{thm:scp}.

\begin{thm}[\cite{georg2013}, Theorem 4.1]
\label{thm:extscp}
For each $i=1,\ldots,M$, choose $\varepsilon_i, \beta_i \in (0,1)$, and fix ${S}_i\geq S(\eps_i,\beta_i,d_i)$ where
\begin{align}
	\begin{split}
	S(\varepsilon_i, \beta_i,d_i):=\min \left\{ N\in\Ne \ \bigg | \ \sum_{j=0}^{d_i-1} \binom{N}{j}\varepsilon_i^j(1-\varepsilon_i)^{N-j} \leq \beta_i\right\}.
	\end{split}
	\label{eq:sample_bound_ext}
\end{align}
Extract $S = \sum_{i=1}^M S_i$ samples i.i.d from $\Delta$ according to a probability measure $\Prb$, construct $\{\Delta_i\}_{i=1}^M$ as in Section \ref{sec:set_up} with $|\Delta_i| = S_i$ and formulate $\scp[\Delta_1,\dots,\Delta_M]$. Under Assumption \ref{ass:feas_uniq}, for each $i=1,\ldots,M$, the minimizer $x^*[\Delta_1,\dots,\Delta_M]$ of $\scp[\Delta_1,\dots,\Delta_M]$ satisfies the chance constraint,
\begin{align}
\ccp_{\eps_i}: \Prb\sqbracks*{g_i\left(x^*[\Delta_1,\dots,\Delta_M],\delta\right) > 0} \leq \eps_i,
	\label{eq:echance_ext}
\end{align}
with confidence (measured with respect to $\mathbb{P}^{S_i}$) at least $1-\beta_i$.
\end{thm}
As with Theorem \ref{thm:scp}, each $d_i$ can be replaced by a tighter upper bound on the support constraints of $g_i$. Equation \eqref{eq:echance_ext} in Theorem \ref{thm:extscp} establishes the feasibility properties of $x^*[\Delta_1,\dots,\Delta_M]$ for each separate constraint. However, no guarantees are provided on the probability that $x^*[\Delta_1,\dots,\Delta_M]$ satisfies all constraints simultaneously, i.e. $\ccp_{\eps}$ in \eqref{eq:echance}. This issue is addressed by the following corollary, that is a direct implication of Theorem \ref{thm:extscp}.
\begin{corol}
\label{corol:product_space_guarantee}
Fix $\eps,\beta \in(0,1)$ and select $\eps_i,\beta_i\in (0,1)$, for all $i=1,\dots,M$, such that $\sum_{i=1}^M\eps_i=\eps$ and $\sum_{i=1}^M\beta_i=\beta$. Under the set-up of Theorem \ref{thm:extscp} and Assumption \ref{ass:feas_uniq} we have that $\Prb^{S}\sqbracks*{x^*[\Delta_1,\dots,\Delta_M]\models \ccp_{\eps}}\allowbreak \geq 1-\beta$, where $\ccp_{\eps}$ is given in \eqref{eq:echance}.
\end{corol}

\begin{proof}
The proof of Corollary \ref{corol:product_space_guarantee} is essentially an application of the Boole-Bonferroni inequalities \cite{prekopa1995stochastic}. By Theorem \ref{thm:extscp} we have that $\Prb^{S}\sqbracks*{x^*[\Delta_1,\dots,\Delta_M]\models \ccp_{\eps_i}}\geq 1-\beta_i,\ \text{for all $i = 1,\dots,M$}$.
By the subadditivity of $\Prb_{\Delta^{S}}$ we have that $\Prb^{S}\sqbracks*{x^*[\Delta_1,\dots,\Delta_M]\models \ccp_{\eps_i}, \text{ for all } i=1,\ldots,M}\geq 1-\sum_{i=1}^M\beta_i=1-\beta.$
To complete the proof it suffices to show that $x^*[\Delta_1,\dots,\Delta_M]\models \ccp_{\eps_i}$ for all $i=1,\ldots,M$, implies that $x^*[\Delta_1,\dots,\Delta_M]\models \ccp_{\eps}$, where $\ccp_{\eps}$ is given in \eqref{eq:echance}. By the subadditivity of $\Prb$, and since $x^*[\Delta_1,\dots,\Delta_M]\models \ccp_{\eps_i}$ is equivalent to  $\Prb\sqbracks*{g_i(x^*[\Delta_1,\dots,\Delta_M],\delta)>0} \allowbreak\leq \eps_i$, we have that $\Prb\sqbracks*{\exists i\in \{1,\ldots,M\} \text{ such that } g_i(x^*[\Delta_1,\dots,\Delta_M],\delta) > 0} \leq \sum_{i=1}^M \eps_i= \eps.$
Since by definition $g(x,\delta):=\max_{i=1,\ldots,M}g_i(x,\delta)$, the last statement implies that $\Prb\sqbracks*{g(x^*[\Delta_1,\dots,\Delta_M],\delta) > 0} \leq \eps$, which is equivalent to $x^*[\Delta_1,\dots,\Delta_M]\models \ccp_{\eps}$ and concludes the proof.
\end{proof}

The computational complexity associated with obtaining $x^*[\Delta_1,\dots,\Delta_M]$ with the feasibility properties of Corollary \ref{corol:product_space_guarantee}, depends on $\{d_i\}_{i=1}^M$ 
and the choices for $\{\varepsilon_i\}_{i=1}^M$,$\{\beta_i\}_{i=1}^M$. The obvious choice of $\varepsilon_i=\varepsilon/M$ and $\beta_i=\beta/M$ for $i=1,\dots,M$ will in general be suboptimal; in Section \ref{sec:complexity} we formulate convex optimization problems to compute better choices. 


\subsection{Recursive scenario approach without re-sampling} \label{sec:mstage_without}
In the sequel we consider $\rcp$ problems with specific structure on the constraint functions that enables us to tackle $\text{SCP}[\Delta_1,\dots,\Delta_M]$ in a sequential manner.
We assume that the constraint functions $g_i (\cdot,\cdot,\cdot):~ \mathcal{X}_i \times \mathcal{X}_{i+1} \times \Delta \rightarrow \mathbb{R}$ are pairwise coupled and convex with respect to their first argument, and $g_M (\cdot,\cdot):~ \mathcal{X}_M \times \Delta \rightarrow \mathbb{R}$. As a consequence of this assumption we have by construction of $\rcp$ that $\mathcal{X} = \mathcal{X}_1 \times \ldots \times \mathcal{X}_M$, which is a special case of the structure assumed in Section \ref{sec:ext_scp}. Let $x = (x_1,\ldots,x_M)$ where $x_i\in\Xs_i$, for each $i=1,\dots,M$. We further assume that the objective function is separable, i.e. $f(x) = \sum_{i=1}^M f_i(x_i)$. Such problem structures appear naturally in SMPC and ADP, as we demonstrate in Section \ref{sec:applications}. The pairwise coupling structure can be relaxed to any form of stage-wise coupling as long as the constraint function at every stage is convex with respect to the decision variables. 


The separable structure assumed, motivates the decomposition of $\text{SCP}[\Delta_1,\dots,\Delta_M]$ into a sequence of coupled scenario programs. For each $i=1,\ldots,M-1$ we define the following parametric scenario program
\begin{align}
\scp_i [x_{i+1},\Delta_i]:
\begin{cases}
	\begin{split}
	\min_{x_i\in\mathcal{X}_i} \quad &f_i(x_i)\\
	 \text{s.t} \quad &g_i\left(x_i,x_{i+1},\delta\right) \leq 0,\ \forall \delta\in\Delta_i
	\end{split}
	\end{cases}
	\label{eq:rob_opt_mstage_samp}
\end{align}
and $\scp_M [\Delta_M]$ analogously, with $g_M(x_M,\delta) \leq 0$ for all $\delta \in \Delta_M$, replacing the corresponding constraint in \eqref{eq:rob_opt_mstage_samp}. We assume that all stage problems in \eqref{eq:rob_opt_mstage_samp} satisfy Assumption \ref{ass:feas_uniq} for any fixed $x_{i+1}\in \Xs_{i+1}$; weaker assumptions are discussed in \cite[Section 4]{compression_paper}.

Consider now the sequence of $M$ pairwise coupled programs $\scp_i\sqbracks*{x_{i+1},\Delta_i}$ and let all sets $\Delta_i$ be identical, i.e. $\Delta_1=\cdots=\Delta_M=\bar{\Delta}$.
Such optimization problems were referred to as cascading programs in \cite{compression_paper}, where the authors study the feasibility properties of a solution generated by sequentially solving a pair of coupled problems using the same set of uncertain scenarios. In particular, the following is a direct consequence of \cite[Theorem 7]{compression_paper}.

\begin{thm}[\cite{compression_paper}, Theorem 7]
\label{thm:recu_scp_wo}
Let $d_i=\dim(\Xs_i)$ be the dimension of the smallest subspace of $\mathbb{R}^d$ containing $\Xs_i$ and $\bar{d}=\sum_{i=1}^M d_i$. Fix $\varepsilon,\beta \in (0,1)$ and $S \geq S(\varepsilon,\beta,\bar{d})$, where $S(\varepsilon,\beta,\bar{d})$ is given by \eqref{eq:sample_bound}. Construct $x^*:=\parens*{x_1^*,\dots,x^*_{M}}$, where each $x^*_i[\bar{\Delta}]$ is recursively computed from (\ref{eq:rob_opt_mstage_samp}) with $\Delta_i=\bar{\Delta}$ for $i=1,\dots,M$. We then have that $\Prb^{S}\sqbracks*{x^*[\bar{\Delta}]\models \ccp_{\eps}}\geq 1-\beta$.
\end{thm}
This recursive scenario based solution can be used to obtain feasibility properties for the solution of each step of the recursion. If $x^*:=\parens*{x_1^*,\dots,x^*_{M}}$ is constructed according to Theorem \ref{thm:recu_scp_wo}, for any fixed $x_{i+1}\in \Xs_{i+1}$, with probability at least $1-\beta_i$, $x^*_i[x_{i+1},\bar{\Delta}]$ satisfies $\Prb \sqbracks*{g_i\parens*{x^*_i[x_{i+1},\bar{\Delta}],x_{i+1},\delta} > 0} \leq \eps_i,$
for any $\varepsilon_i,\beta_i\in (0,1)$ satisfying the equation $S(\eps_i,\beta_i,d_i)\leq S$ where $S$ is chosen such that  $S\geq S(\eps,\beta,\bar{d})$. In this case, however, the values of $\varepsilon_i$ and $\beta_i$ are not set a-priori and are not design choices; they are implicitly determined by the dimension of each subproblem. Consequently, the computational complexity of the recursive scenario approach only depends on $\bar{d}$ and the choice of $\eps,\beta$.

\subsection{Recursive scenario approach with re-sampling}
\label{sec:mstage}
Consider the separable structure assumed in Section \ref{sec:mstage_without} and note that for a fixed $x_{i+1}\in \Xs_{i+1}$, $\scp_i[x_{i+1},\Delta_i]$ is in the form of $\scp[\bar{\Delta}]$ considered in Section \ref{sec:scp}. Fix $\varepsilon_i$ and $\beta_i$ and let the number of samples $S_i$, $i=1,\ldots,M$ be chosen according to \eqref{eq:sample_bound_ext}. For any $x_{i+1}\in \Xs_{i+1}$, $\Delta_i\in \Delta^{S_i}$, let $x^*_i [x_{i+1},\Delta_i] :~ \Xs_{i+1}\times\Delta^{{S}_i}\rightarrow \Xs_i$ be the minimizer of $\scp_i[x_{i+1},\Delta_i]$. Theorem \ref{thm:scp} implies that for all $i=1,\ldots,M-1$, $x^*_i [x_{i+1},\Delta_i]$ satisfies the chance constraint
\begin{align}
\ccp_{\eps_i}[x_{i+1}]:	 \Prb\sqbracks*{g_i\parens*{x^*_i[x_{i+1},\Delta_i],x_{i+1},\delta} > 0} \leq \eps_i,
	\label{eq:echance_stage}
\end{align}
with probability at least $1 - \beta_i$, while for $i=M$, $x_M^*[\Delta_M]$ satisfies $\ccp_{\varepsilon_M}$ with probability at least $1-\beta_M$.

Using the parametrized scenario optimization problems in the form of (\ref{eq:rob_opt_mstage_samp}) we can recursively construct a decision vector $x^*:=\parens*{x_1^*,\dots,x^*_{M}}$, where for each $i=1,\ldots,M-1$ the optimizer $x^*_i [x_{i+1},\Delta_i]$ can be written as $x^*_i[\Delta_i,\dots,\Delta_{M}]:~ \Delta^{S_{i}} \times \cdots \times \Delta^{S_{M}}\rightarrow \Xs_{i}$, satisfying
\begin{align}
\Prb^{S_{i}}\sqbracks*{x^*_{i}[\Delta_{i},\dots,\Delta_{M}]\models \ccp_{\eps_{i}}\Big [x^*_{i+1}[\Delta_{i+1},\dots,\Delta_{M}] \Big ] }\geq 1-\beta_{i}, \label{eq:ebchance_resampling}
\end{align}
and for $i=M$, $\Prb^{S_{M}} \sqbracks*{x^*_M \models \ccp_{\varepsilon_M}} \geq 1-\beta_M$. Note that due to the recursive process, $x_i^*$ depends implicitly on all sets $\Delta_i, \ldots, \Delta_M$. The following theorem can be used to compare the feasibility properties of a solution constructed in this way with a solution obtained using Theorems \ref{thm:scp},\ref{thm:extscp} and \ref{thm:recu_scp_wo}.
\begin{thm}
\label{thm:recursive_scp}
Fix $\eps,\beta \in(0,1)$ and select $\eps_i,\beta_i\in (0,1)$, for $i=1,\dots,M$, such that $\sum_{i=1}^M\eps_i=\eps$ and $\sum_{i=1}^M\beta_i=\beta$. Construct $x^*:=\parens*{x_1^*,\dots,x^*_{M}}$, where each $x^*_i[\Delta_i,\dots,\Delta_{M}]$ is recursively computed from (\ref{eq:rob_opt_mstage_samp}), satisfying \eqref{eq:ebchance_resampling}.
Let $S=\sum_{i=1}^{M}{S_i}$ with $\{S_i\}_{i=1}^M$ chosen according to \eqref{eq:sample_bound_ext}. We then have that $\Prb^{S}\sqbracks*{x^*[\Delta_{1},\dots,\Delta_{M}]\models \ccp_{\eps}}\geq 1-\beta$.
\end{thm}
\begin{proof}
Let $\bar{S}_i=\sum_{k=i}^M S_k$, $\bar{\eps}_i=\sum_{k=i}^M \eps_k$, $\bar{\beta}_i=\sum_{k=i}^M \beta_k$, $\bar\Delta_i=(\Delta_i,\dots,\Delta_M)$ and $\bar{x}^*_i=\parens*{x_i^*,\dots,x^*_{M}}$.
We claim that for all $i=1,\ldots,M$, the following statement holds
\begin{align}
\label{eq:ind_stmt}
\begin{split}
&\Prb^{\bar{S}_{i}} \Big [ \Prb \big [ g_{M}\parens*{x^*_M[\bar{\Delta}_M],\delta}>0  \text{ or } \\
& \exists k \in \mathbb{N}_+, i \leq k < M :~ g_{k}\parens*{x^*_k[\bar{\Delta}_k],x^*_{k+1}[\bar{\Delta}_{k+1}],\delta}>0 \big ] \leq \bar{\eps}_i \Big ]\geq 1- \bar{\beta}_i. 
\end{split}
\end{align}
The statement of the claim implies that, with confidence at least $1- \bar{\beta}_i$, $\bar{x}^*_i$ satisfies all constraints with indices greater than or equal to $i$, with probability at least $1-\bar{\eps}_i$.

If the claim holds, then for $i=1$ we get the result. We show that the claim holds using induction.
For $i=M$, \eqref{eq:ind_stmt} is trivially satisfied since $\bar{\Delta}_M = \Delta_M$ and $\scp_M [\Delta_M]$ is in the form of $\scp[\bar{\Delta}]$ considered in Section \ref{sec:scp} with $\bar{\Delta}, S, \eps, \beta$ replaced by $\Delta_M, S_M, \eps_M$ and $\beta_M$, respectively. Assume that \eqref{eq:ind_stmt} holds for some $1<i<M$. By \eqref{eq:ebchance_resampling} we have that
\begin{align}
\label{eq:ind_stmt_proof}
\begin{split}
&\Prb^{S_{i-1}} \Big [ \Prb\big [ g_{i-1}\parens*{x^*_{i-1}[\bar{\Delta}_{i-1}],x^*_{i}[\bar{\Delta}_i],\delta}>0 \big ] \leq \eps_{i-1} \Big ] \geq 1- \beta_{i-1}. 
\end{split}
\end{align}
Using the fact that all samples are extracted independently, and \eqref{eq:ind_stmt}, \eqref{eq:ind_stmt_proof}, hold for any uncertainty realization not in $\bar{\Delta}_{i}$ and $\bar{\Delta}_{i-1}$, respectively, \eqref{eq:ind_stmt}, \eqref{eq:ind_stmt_proof} would also hold with
$\Prb^{\bar{S}_{i-1}}$ in place of $\Prb^{\bar{S}_{i}}$ and $\Prb^{S_{i-1}}$. From the resulting statements and the subadditivity of $\Prb^{\bar{S}_{i-1}}$ and $\Prb$, we can then show analogously to the proof of Corollary \ref{corol:product_space_guarantee} that \eqref{eq:ind_stmt}
holds with $i-1$ in place of $i$.
The latter implies that $\Prb^{\bar{S}_{i-1}} \sqbracks*{ \bar{x}^*_{i-1} \models \ccp_{\bar{\eps}_{i-1}} } \geq 1- \bar{\beta}_{i-1}$ and proves the claim.
\end{proof}

The computational complexity associated with obtaining $x^*[\Delta_1,\dots,\Delta_M]$ with the feasibility properties of Theorem \ref{thm:recursive_scp} depends on $\{d_i\}_{i=1}^M$ and the choices for $\{\varepsilon_i\}_{i=1}^M$,$\{\beta_i\}_{i=1}^M$. Notice that unlike the recursive scenario approach without re-sampling, $\{\eps_i\}_{i=1}^M$ and $\{\beta_i\}_{i=1}^M$ are again design parameters for $i=1,\dots,M$ and can be chosen in a way that reduces the computational complexity of the algorithm used to solve the corresponding optimization problems. We deal with this issue in the next section.

\subsection{Complexity optimization} \label{sec:complexity}
For the problems in Sections \ref{sec:scp} and \ref{sec:mstage_without}, the number of decision variables $d$ and $\sum_{i=1}^M d_i$ and the overall violation and confidence levels $\varepsilon$ and $ \beta$ determine the total complexity. For Sections \ref{sec:ext_scp} and \ref{sec:mstage} on the other hand, although the overall violation and confidence are chosen a priori, the stage-wise levels $\{\eps_i\}_{i=1}^M$ and $\{\beta_i\}_{i=1}^M$ are typically not fixed by the problem data and constitute a design choice that can affect the computational complexity due to the cubic dependence of SOCP solvers on the total number of samples and decision variables. Since the values of $\{d_i\}_{i=1}^M$ are fixed by problem data and generating samples from $\Delta$ can be hard, we focus on minimizing the total number of samples as an approximation to minimizing the total complexity. Throughout this section we replace the implicit sample size bound $S(\varepsilon_i, \beta_i,d_i):=\min \left\{ N\in\Ne \ \bigg | \ \sum_{j=0}^{d_i-1} \binom{N}{j}\varepsilon_i^j(1-\varepsilon_i)^{N-j} \leq \beta_i\right\}$ that upper bounds the required sample size by the explicit bound $S(\varepsilon_i, \beta_i,d_i)\geq \tfrac{e}{e-1} \tfrac{1}{\eps_i}\parens*{d_i-1+\ln\parens*{\tfrac{1}{\beta_i}}}$ due to \cite{alamo2010sample}. For simplicity we treat the right-hand-side as an integer.

\begin{prop}
Consider the setup of Sections \ref{sec:ext_scp} and \ref{sec:mstage} where for each $i=1,\ldots,M$ the values of $d_i=\dim\parens*{\Xs_i}$ and $d=\dim\parens*{\Xs}$ are fixed by the problem data. Fix $\varepsilon, \beta \in (0,1)$. The problem of selecting $\{\varepsilon_i,\beta_i \in (0,1)\}_{i=1}^M$ with $\sum_{i=1}^{M}\eps_i\leq \varepsilon,\ \sum_{i=1}^{M}\beta_i\leq \beta$ that minimize the total number of samples $\sum_{i=1}^M{S(\varepsilon_i, \beta_i,d_i)}$, is a convex optimization program of the form:
\begin{align}
\label{eq:min_complexity_ext}
	\begin{split}
	\min_{\{\eps_i,\beta_i\}_{i=1}^M}\quad &\sum_{i=1}^{M} S(\varepsilon_i,\beta_i,d_i)\\
	\text{subject to:}		
		\quad &\sum_{i=1}^{M}\eps_i\leq \varepsilon,\sum_{i=1}^{M}\beta_i\leq \beta,\eps_i,\beta_i>0, \forall i\in\{1,\dots,M\}.\\
	\end{split}	
\end{align}
\label{prop:min_complex_ext}
\end{prop}
\begin{proof}

The function $S(\eps_i,\beta_i,d_i)$ is convex with respect to $\eps_i,\beta_i$ since the Hessian matrix is positive definite for any $\eps_i,\beta_i\in (0,1)$. As a result, $\sum_{i=1}^M{S(\eps_i,\beta_i,d_i)}$ is the sum of convex functions.
\end{proof}
The objective function of problem \eqref{eq:min_complexity_ext} is not in a standard form compatible with commercially available optimization software. As a result, one needs to implement a first or second order method to solve \eqref{eq:min_complexity_ext} (see for example \cite{nesterov2013universal}) taking advantage of the fact that both the gradient and Hessian matrix of the objective function are bounded with respect to $\varepsilon_i,~ \beta_i$ in $[\mu,1)$ for any $\mu>0$. Fixing the confidence levels $\beta_i$ a priori (e.g. $\beta_i =\beta/M$) simplifies the structure of (\ref{eq:min_complexity_ext}) significantly and transforms the problem into a standard semi-definite program (SDP).

\begin{prop}
Choose $\beta\in (0,1)$ and fix the stage-wise confidence levels $\{\beta_i \in (0,1)\}_{i=1}^M$ such that $\sum_{i=1}^M\beta_i\leq \beta$. Fix $\varepsilon \in (0,1)$. For $c_i=\tfrac{e}{e-1} \parens*{d_i-1+\ln\parens*{\tfrac{1}{\beta_i}}}$, $i=1,\ldots,M$, the following SDP is equivalent to \eqref{eq:min_complexity_ext}.
\begin{align}
	\begin{split}
	\min_{\{t_i,\eps_i\}_{i=1}^M}\quad &\sum_{i=1}^{M}t_i\\
	\text{subject to:} &\bM t_i & \sqrt{c_i}\\\sqrt{c_i} & \eps_i \eM \succcurlyeq 0, \sum_{i=1}^{M}\eps_i\leq\eps, \eps_i>0, \forall i\in\{1,\dots,M\}\\
	\end{split}
	\label{eq:min_samples_eps}
\end{align}
\end{prop}

\begin{proof}
The objective function in \eqref{eq:min_complexity_ext} can be written as $\sum_{i=1}^M c_i/\eps_i$. Writing the problem in standard epigraph form and using Schur's complement we end up with the constraints in \eqref{eq:min_samples_eps}.
\end{proof}


\begin{table*}[tp]
\scriptsize
\centering
\caption{Complexity characteristics of the methods presented in Sections \ref{sec:scp}- \ref{sec:mstage}.}
\begin{tabular}{l|l|l|l|l}
                       & Section \ref{sec:scp} &Section \ref{sec:ext_scp} & Section \ref{sec:mstage_without}  &  Section \ref{sec:mstage}\\ \hline
														Number of problems & 1 &1 &  $M$ & $M$ \\
Samples per problem & $S$ $\sim$ \eqref{eq:sample_bound} & $S_i$ $\sim$ \eqref{eq:sample_bound_ext}& $S$ $\sim$ \eqref{eq:sample_bound}, $d = \sum_{i=1}^M d_i$& $S_i$ $\sim$ \eqref{eq:sample_bound_ext}\\

Total number of samples    &      $S$              &          $\sum_{i=1}^M{S_i}$                 &                $S$                      &              $\sum_{i=1}^M{S_i}$                     \\
Decision variables per problem          &   $d=\dim\parens*{\Xs}$                &    $d=\dim\parens*{\Xs}$                         &               $ d_i=\dim\parens*{\Xs_i}$                       &              $ d_i=\dim\parens*{\Xs_i}$                    \\
Constraints per problem&    $ MS$               &      $ \sum_{i=1}^M{S_i}$                       &                 $ S$                    &             $ S_i$                        \\ \hline \hline
\textbf{Total complexity (SOCP)}   &     $\mathcal{O}\parens*{\parens*{d+MS}^3}$              &    $\mathcal{O}\left((d+ \sum_{i=1}^M{S_i})^3\right)$                         &                 $\sum_{i=1}^{M}\mathcal{O}\parens*{\parens*{d_i+S}^3}$                &         $\sum_{i=1}^{M}\mathcal{O}\parens*{\parens*{d_i+S_i}^3}$                             \\
\hline
\end{tabular}
\label{tab:all_complexities}
\vspace{-2\baselineskip}
\end{table*}

\section{Discussion and trade-offs} \label{sec:trade}
Each scenario based algorithm presented in Section \ref{sec:gen} assumes a specific structure on the original $\rcp$ to construct a probabilistically feasible solution. Here we discuss differences between the feasibility properties of each solution and analyze the computational complexity of the associated algorithms as a function of design parameters.

\subsection{Structure and feasibility properties}
In contrast to the standard scenario approach of Section \ref{sec:scp}, the multi-stage variant of Section \ref{sec:ext_scp} assumes that the domain of each constraint function in $\rcp$ is restricted to a subset of $\Xs$. By investigating each constraint separately, Theorem \ref{thm:extscp} provides guarantees on the probability that $x^*[\Delta_1,\ldots,\Delta_M]$ satisfies every individual constraint, something that cannot be achieved with the standard scenario approach. In the recursive methodologies of Sections \ref{sec:mstage_without} and \ref{sec:mstage} we further restrict the structure of $\rcp$ by requiring the constraint functions to be pairwise coupled. In this way we relax the assumption regarding the convexity of the constraint functions. In particular, we require $g_i(x_i,x_{i+1},\delta)$ to be convex with respect to $x_i$, but do not require any convexity assumptions for the dependance on $x_{i+1}$. One situation where this can be of advantage is optimization programs with constraint functions that are bi-convex with respect to two decision vectors. Practically, such problems are often solved through a descent algorithm, alternating between optimizing with respect to one of the decision vectors while fixing the other decision vector to the value obtained at the preceding iteration. Theorem \ref{thm:recu_scp_wo} allows us to provide probabilistic guarantees for the feasibility of the solution generated through such a descent algorithm, provided we fix a priori the number of iterations considered. Moreover, using the methodology of Section \ref{sec:mstage_without} which employs the same samples at every step of the recursive methodology ensures monotonicity of the objective function between consecutive steps of the recursion, that is crucial to ensure that the objective function decreases; see \cite[Section 4]{compression_paper}. 


Theorem \ref{thm:scp}, Corollary \ref{corol:product_space_guarantee} and Theorems \ref{thm:recu_scp_wo} and \ref{thm:recursive_scp} all lead to a feasibility statement in the form of $\Prb^S\sqbracks*{x^*[\Delta_{1},\dots,\Delta_{M}]\models \ccp_{\eps}}\geq 1-\beta$. Each method however requires a different number of samples to construct a solution and in turn the space on which the confidence related to the probability of constraint satisfaction is measured differs. In the standard scenario approach the total number of samples $S$ is determined by the value of $d$ and the choice of violation and confidence levels $\varepsilon,\beta$ (inspect \eqref{eq:sample_bound}). Assuming the same choice of $\varepsilon$ and $\beta$, the total number of samples in the multi-stage scenario approach $\sum_{i=1}^M S_i$ can be greater or less than $S$ depending on the values of $\{d_i\}_{i=1}^M$ (inspect \eqref{eq:sample_bound_ext} and the first two columns in Table \ref{tab:all_complexities}). In general, if each $d_i$ is significantly smaller than $d$, then the total number of samples is smaller in the multi-stage scenario approach. The situation is analogous between the recursive scenario approach without and with re-sampling, where the total number of samples will be generally higher in the latter depending on the values of $\{d_i\}_{i=1}^M$ and the choices of $\{\eps_i\}_{i=1}^M$,$\{\beta_i\}_{i=1}^M$ (see the last two columns in Table \ref{tab:all_complexities}). Note that for the multi-stage scenario approach and the recursive scenario approach with re-sampling we can use the methods of Section \ref{sec:complexity} to optimize over $\{\eps_i\}_{i=1}^M$ and $\{\beta_i\}_{i=1}^M$ but there is no guarantee that this will lead to a smaller number of total samples since $\{d_i\}_{i=1}^M$ is fixed by problem data.

\subsection{Complexity}


Both the standard and multi-stage scenario approach of Sections \ref{sec:scp} and \ref{sec:ext_scp} require solving a single problem of the same structure with the same number of decision variables but a potentially different number of constraints. The number of decision variables $d$ is given by problem data, while the number of constraints depends on $d$, $\{d_i\}_{i=1}^M$ and the chosen $\varepsilon,\beta$ and $\{\eps_i\}_{i=1}^M$, $\{\beta_i\}_{i=1}^M$. In the standard scenario approach, we use the same samples $S$ (see \eqref{eq:sample_bound}) for each constraint function leading to a total of $MS$ constraints. In the multi-stage scenario approach, we use different samples $S_i$ (see \eqref{eq:sample_bound_ext}) for each constraint function leading to a total of $\sum_{i=1}^M{S_i}$ constraints, a number that can be minimized over $\{\eps_i\}_{i=1}^M$, $\{\beta_i\}_{i=1}^M$ using the methods of Section \ref{sec:complexity}. The computational complexity of each method is reported in the first two columns of Table \ref{tab:all_complexities}. Depending on the ratio between the minimum value of $\sum_{i=1}^M{S_i}$ and $MS$, either of the two methods might be preferable.



The computational complexity of the recursive methodologies of Sections \ref{sec:mstage_without} and \ref{sec:mstage} depends on the number of decision variables and constraints per subproblem. Each subproblem involves a single constraint function and as a result the number of samples required by Theorems \ref{thm:recu_scp_wo} and \ref{thm:recursive_scp} coincides with the number of constraints. In the recursive scenario approach without re-sampling, we use the same number of samples $S$ in every subproblem which depends on $\bar{d}=\sum_{i=1}^M{d_i}$ and the choice of $\eps,\beta$ (see Theorem \ref{thm:recu_scp_wo}). If $\bar{d}=d$ (as is the case for example in some ADP problems, see Section \ref{sec:applications}), the number of samples coincides with that of the standard scenario approach. In general however, it might very well be that $\bar{d}>d$ (as is the case, for example, in some SMPC problems). In the recursive scenario approach with re-sampling, the number of samples $S_i$ in each subproblem coincides with the number of samples used in the multi-stage scenario approach and depends on $d_i$ and the choice of $\eps_i$, $\beta_i$  (see Theorem \ref{thm:recursive_scp}). As in the multi-stage scenario approach, $\sum_{i=1}^M{S_i}$ can be minimized over $\{\eps_i\}_{i=1}^M$, $\{\beta_i\}_{i=1}^M$ using the methods in Section \ref{sec:complexity}. The computational complexity of both recursive methods is reported in the last two columns of Table \ref{tab:all_complexities}. Whenever applicable, the recursive methods of Sections \ref{sec:mstage_without} and \ref{sec:mstage} can provide significant computational advantages, as illustrated in the next section.

\section{Numerical example: Approximate Dynamic Programming} \label{sec:applications}
Dynamic programming (DP) recursions are widely used to characterize the value function of optimal control problems \cite{bertsekas1995dynamic}. For systems with continuous states, explicitly computing the value function by space discretization methods suffers from the curse of dimensionality, making the process intractable for state spaces of even moderate dimensions. This has motivated the development of sophisticated ADP methods \cite{powell2007approximate}. A recently established methodology is the linear programming approach to ADP \cite{de2003linear} which projects the optimal value function on the span of a pre-selected set of basis functions, intersected with the feasibility region determined by a set of inequality constraints. The authors in \cite{kariotoglou2014adp,kariotoglou2013approximate} developed an algorithm based on the linear programming approach to ADP, specifically to approximate the value function of stochastic reachability problems. In this section we use this algorithm to investigate the relative performance of the alternative scenario program formulations of Section \ref{sec:gen}. We consider a simplified planar unicycle model with additive noise 
\begin{align}
\begin{bmatrix}\delta_1(i+1)\\\delta_2(i+1) \end{bmatrix}=\begin{bmatrix}
						\delta_4(i)\cos(\delta_3(i))+\delta_1(i)  \\
						\delta_4(i)\sin(\delta_3(i))+\delta_2(i)  \
						\end{bmatrix}
						+w_i
						\label{eq:system}
\end{align}
where $\delta_1, \delta_2$ denote linear position, $\delta_3$ yaw angle and $\delta_4$ linear velocity. We assume that $\delta_3$ and $\delta_4$ are control inputs to the system and treat $\delta_1$ and $\delta_2$ as states. The noise terms $w_i\in \mathbb{R}^2$ are assumed to be independent for different $i$ and identically distributed according to a multivariate normal distribution $\mathcal{N}(0,\Sigma)$ with diagonal covariance matrix. The combined state-action space is denoted by $\Delta=\Delta_x \times \Delta_u =\mathbb{R}^2\times \left([-0.5,0.5]\times[-2\pi,2\pi]\right)$ where for $\delta=(\delta_1,\delta_2,\delta_3,\delta_4)=(\delta_x,\delta_u)\in\Delta$, $\delta_x$ corresponds to spatial coordinates while $\delta_u$ to control inputs. The symbol $\delta$ is used for the state and input variables since in the sequel we will be sampling from $\Delta$. Given a target set $T=[0.8,1]^2$, an avoid set $A=[-0.45,0.25]\times[-0.2, 0.15]$ and a collection of time indexed safe sets $\{S_i\}_{i=1}^3=\bracks*{[-1,1]^2,[-0.3,1]^2,[0.4,1]^2}$, the three step reach-avoid problem considered here is to maximize the probability that \eqref{eq:system} reaches $T$ while staying in the corresponding safe region $S_i\setminus A$ for time steps $i=1,2,3$ (see Figure \ref{fig:RAvf}). The authors in \cite{summers2010verification} show that this reach-avoid problem can be solved via a DP recursion:
\begin{align}
\label{eq:DPrecu}
\begin{split}
&V_i^*(\delta_x)=\sup_{\delta_u\in \Delta_u}\{\underbrace{\mathds{1}_T(\delta_x)+\mathds{1}_{(S_i\setminus A )\setminus T}(\delta_x)\int_{\Delta_x}{V^*_{i+1}(y)Q(dy|\delta)}}_{h(\delta_x,\delta_u)}\}\\
&V_{4}^*(\delta_x)=\mathds{1}_T(\delta_x).
\end{split}
\end{align}
where $V_i^*$ denotes the value function at stage $i$, $Q$ denotes the transition kernel of the stochastic process in \eqref{eq:system} and $\mathds{1}_T,\mathds{1}_{(S_i\setminus A )\setminus T}$ denote the indicator functions of the sets $T$ and $(S_i\setminus A )\setminus T$ respectively. We follow the ADP formulation for reach-avoid problems suggested in \cite{de2003linear} and applied in \cite{kariotoglou2014adp,kariotoglou2013approximate} to approximate \eqref{eq:DPrecu}. To this end we express the value function of each step in the DP recursion as a solution to an infinite dimensional linear program:
\begin{align}
\label{eq:inflp}
\begin{split}
V_i^* \in \arg&\inf_{V(\cdot)\in \mathcal{F}} \quad \int_{\Delta_x} V(\delta_x) \nu (\text{d$\delta_x$)} \\
	&\text{subject to} \quad  V(\delta_x)\geq h(\delta_x,\delta_u) , \ \forall \delta\in \Delta
\end{split}
\end{align}
where $\nu$ is a (positive) measure supported on $\Delta_x$ and $\mathcal{F}$ denotes the space of Borel-measurable functions in which, under mild assumptions, $V^*_i$ resides \cite{kariotoglou2014adp}. Problems in the form of \eqref{eq:inflp} are generally intractable and it is common in the literature to restrict the decision space to a finite dimensional subspace of $\mathcal{F}$ to approximate each $V^*_i$.
As suggested in \cite{kariotoglou2014adp}, we restrict the decision space to a set of Gaussian radial basis functions (RBFs) with fixed centers and variances and use their span to approximate each $V_i^*$. Let $\{d_i\}_{i=1}^3=\{200,150,100\}$ denote the cardinality of each basis set over the time horizon and $x=\{x_i\}_{i=1}^3$ with $x_i\in \mathbb{R}^{d_i}$, a collection of vectors corresponding to the weights of each RBF in the set. The reduction in the number of basis elements over the horizon is motivated by the reduction in the size of each safe set $S_i$. We denote by $L_{i}: \mathbb{R}^{d_i} \times \mathbb{R}^{d_{i+1}}\times \Delta \rightarrow \mathbb{R}$ the functions (linear in the first and second arguments) that for $i=1,2$ and each $\delta\in\Delta$ return the difference between the approximate value function at time $i$ and the one-step-ahead reward at time $i+1$ (observe the constraints in \eqref{eq:inflp}).  Each $L_i$ implicitly depends on the safe, avoid and target regions at time $i$ and the weights $x_i,x_{i+1}$ completely determine its value over $\Delta$. For $i=3$, the function is defined as $L_{3}: \mathbb{R}^{d_3} \times \Delta \rightarrow \mathbb{R}$ since the reach-avoid value function at $i=4$ is known \eqref{eq:DPrecu}. Using this notation, the approximate reach-avoid value functions can be computed via a sequence of coupled robust linear programs:
\begin{align}
\label{eq:finite_basis_opt_mult}
\begin{split}
\min_{x_i\in \mathbb{R}^{d_i}} \quad &x_i^\top I_i\\
\text{subject to:} \quad &L_i(x_i,x_{i+1},\delta)\geq 0, \ \forall \delta\in \Delta
	\end{split}
\end{align}
where $I_i$ denotes the element-wise integral over $\Delta$ of each RBF in the basis set with respect to the measure $\nu$. The sequence of problems in \eqref{eq:finite_basis_opt_mult} can be combined to a single problem as:
\begin{align}
\label{eq:finite_basis_opt_single}
\begin{split}
\min_{x\in\mathbb{R}^{d_1+d_2+d_3}} \quad &\sum_{i=1}^3 x_i^\top I_i\\
\text{subject to:} \quad &L_i(x_i,x_{i+1},\delta)\geq 0, \ \forall \delta\in \Delta, \ i=1,2 \\
\quad &L_3(x_3,\delta)\geq 0, \ \forall \delta\in \Delta. \\
	\end{split}
\end{align}
Using the methods presented in Section \ref{sec:gen} to solve \eqref{eq:finite_basis_opt_mult} and \eqref{eq:finite_basis_opt_single} we can obtain an optimal solution for the weight vector $x$, possibly different for each method. Using the optimal weights, we can then directly construct the approximate value function of the stochastic reach-avoid problem for each $i=1,2,3$.

We solved the problem with all methods and Table \ref{tab:adp_results} compares the theoretical feasibility guarantees (column $\varepsilon$) with the empirical ones (column $\hat{\varepsilon}$) along with the associated complexities (columns ``Sampling'' and ``Solver''). The empirical violation values were calculated by uniformly sampling 1000 realizations from $\Delta$, other than those used in the optimization process, and computing the ratio between the number of realizations that resulted in constraint violation and 1000. We highlight with bold the parameters that can be chosen by the user and are not fixed by the problem data; for the multi-stage scenario approach and the recursive scenario approach with re-sampling, we have chosen the violation levels $\varepsilon_i$ at each stage by solving the complexity optimization program in \eqref{eq:min_samples_eps}. The associated confidence levels $1-\beta_i, i=1,2,3$ were all fixed to $0.99$ to achieve an overall confidence $1-\beta$ of at least $0.97$. The basis centers and variances were sampled uniformly at random from each safe set and $(0,0.01]$ respectively. All computations were done on an Intel Core i7 Q820 CPU clocked @1.73 GHz with 16GB of RAM memory, using the Gurobi optimization suite. Figure \ref{fig:RAvf} shows the level sets of the approximation at time $i=1$ restricted on $[-1,1]^2$, constructed using the recursive scenario approach with re-sampling. Even though the optimal value function corresponds to a reach-avoid probability, the values of the approximation go above 1 since it is only an upper bound \cite{kariotoglou2014adp}. The accuracy of the approximation can be increased by increasing the number of basis elements or reducing the values of $\varepsilon, \beta$.

\begin{figure}[t]
\centering
\includegraphics[width=7cm, height=5cm]{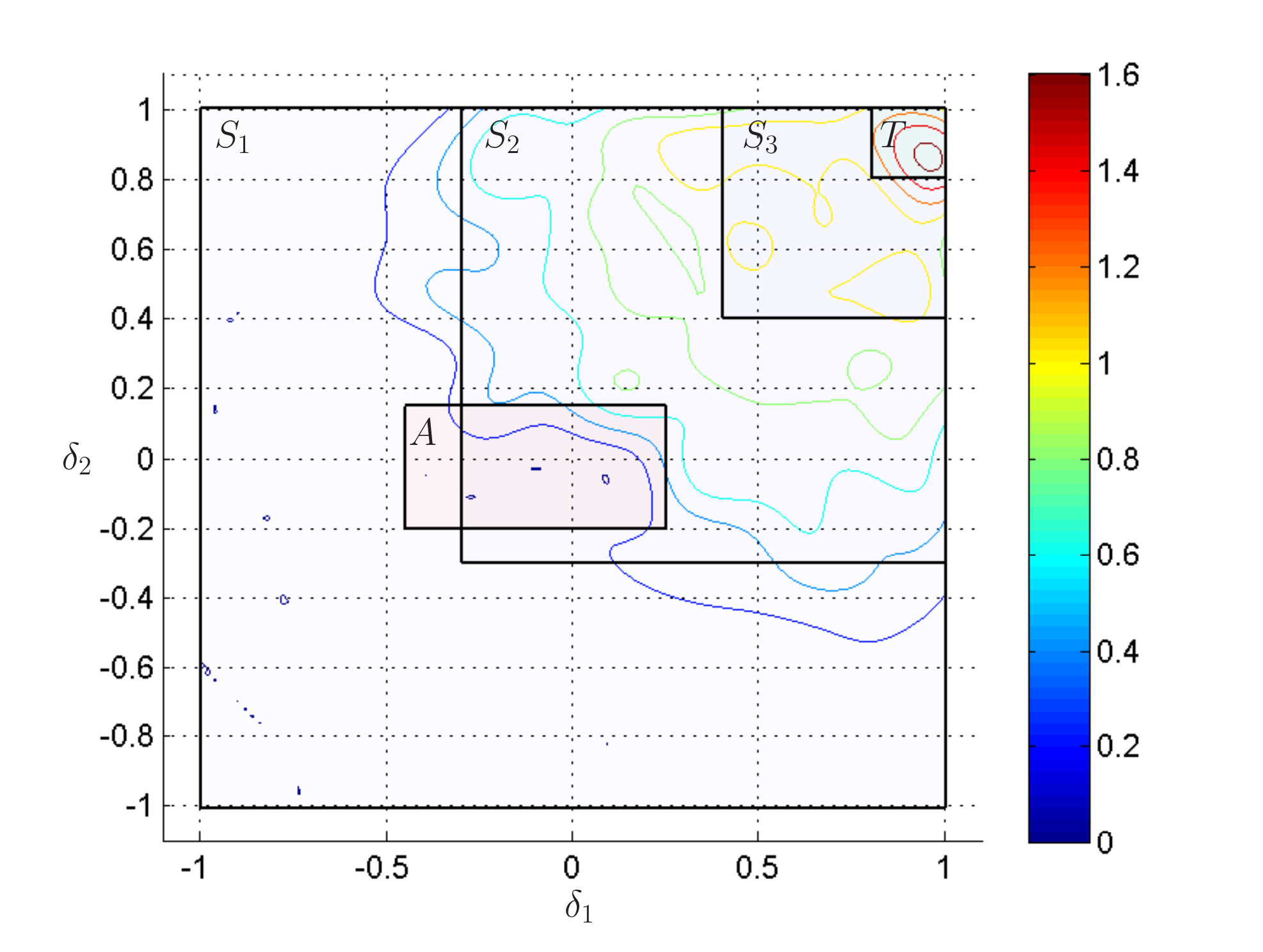}%
\vspace{-1\baselineskip}
\caption{Level sets of the approximate value function at $i=1$ restricted on $[-1,1]^2$, constructed using the method in Section \ref{sec:mstage}.}
\label{fig:RAvf}%
\end{figure}

\begin{table*}[tp]
\scriptsize
\centering
\caption{Results of ADP for reach-avoid using the methods presented in Sections \ref{sec:scp}- \ref{sec:mstage}}
\begin{tabular}{llllllrrr}
&Horizon step&$\eps$&$\hat{\eps}$&$1-\beta$&$d$&Constraints&Solver (sec)&Sampling (sec)\\ \hline \hline
\multirow{4}{*}{Section \ref{sec:scp}}&i=3 &-&0.045&-&-&-&-&-\\
																			&i=2 &-&0.034&-&-&-&-&-\\
																			&i=1 &-&0.028&-&-&-&-&-\\																			
																			&Overall &\textbf{0.1}&0.038&\textbf{0.97}&450&43056&71&0.852\\
																			\hline \hline

\multirow{4}{*}{Section \ref{sec:ext_scp}}&i=3 &\textbf{0.028}&0.027&\textbf{0.99}&100&11865&-&0.8613\\
																			&i=2 &\textbf{0.034}&0.019&\textbf{0.99}&150&14447&-&0.7786\\
																			&i=1 &\textbf{0.039}&0.028&\textbf{0.99}&200&16632&-&0.7411\\
																			&Overall &\textbf{0.1}&0.065&\textbf{0.97}&450&42944&90&2.38\\
																			\hline \hline

\multirow{4}{*}{Section \ref{sec:mstage_without}}&i=3 &-&0.014&-&100& 14352&1.78&-\\
																			&i=2 &-&0.027&-&150& 14352&4.7&-\\
																			&i=1 &-&0.035&-&200& 14352&8.82&-\\
																			&Overall &\textbf{0.1}&0.069&\textbf{0.97}&-&-&15.3&0.7671\\
																			\hline \hline
\multirow{4}{*}{Section \ref{sec:mstage}}&i=3 &\textbf{0.028}&0.009&\textbf{0.99}&100&11865&1.78&0.889\\
																			&i=2 &\textbf{0.034}&0.016&\textbf{0.99}&150&14447&2.65&0.8242\\
																			&i=1 &\textbf{0.039}&0.021&\textbf{0.99}&200&16632&4.76&0.6364\\
																			&Overall &\textbf{0.1}&0.045&\textbf{0.97}&-&-&9.19&2.35\\
\hline \hline
\end{tabular}
\label{tab:adp_results}
\vspace{-2\baselineskip}
\end{table*}

The results in Table \ref{tab:adp_results} indicate that in this instance it is favorable to solve problems in a recursive manner since the same overall violation levels are respected  while the computation times are smaller. Notice that in the standard scenario approach, the total number of samples is three times smaller than the total number of constraints since every sample is enforced on every constraint function in the horizon separately. Moreover, in the multi-stage scenario approach and the recursive scenario approach with re-sampling we have to generate different samples for each constraint function in the horizon and thus sampling consumes more time. The reported solver times differ since differences in the sampled data affect solution time. In particular, the samples used for each of the constraints of the standard scenario approach are identical, giving structure to the problem which appears to be exploited by the solver. For the multi-stage scenario approach different samples are used for each constraint and the resulting optimization program has less structure. The differences in the reported sampling times (even when sample numbers are the same) are a consequence of the hit and run algorithm used to generate them \cite{kroese2011handbook}. The numbers reported are averaged over 10 runs of each method.


\section{Conclusion}
We investigated the feasibility properties of different scenario based optimization programs, involving the standard scenario approach, its multi-stage counterpart as well as recursive variants that can be employed in case the problem exhibits a separable structure. We showed how confidence and violation levels can be treated as optimization assets and can be selected by means of convex optimization problems to reduce the computation time of the associated algorithm. We verified with a numerical example that the recursive structure often encountered in sequential decision making can be exploited, leading to much shorter computation times.

Our future work focuses on utilizing the insights gained in this paper in different problems where the assumed recursive structure is present. We already demonstrated the relevance and benefit of this in a class of approximate dynamic programming algorithms and believe that similar computational advantages will be observed in stochastic model predictive control problems. We also believe that recursive structures appear naturally in multi-agent systems where the decisions of one agent depend on the decision of another; in such cases using different samples between agents can have a significant impact on the required communication bandwidth. In terms of applications, we intend to use the recursive scenario approach discussed here to address surveillance tasks that are posed as reach-avoid problems \cite{kariotoglou2014camaeras}.
\label{sec:conclusion}

\bibliographystyle{IEEEtran}
\bibliography{IEEEabrv,ifamult}
\end{document}